\documentclass[12pt]{article}

\usepackage{graphics,amsmath,amssymb,amsthm,mathrsfs}

\newtheorem{thm}{Theorem}[section]
\newtheorem{lemma}[thm]{Lemma}

\theoremstyle{definition}
\newtheorem{remark}[thm]{Remark}

\newcommand{\varep}{\varepsilon}

\numberwithin{equation}{section}

\begin{document}

\bibliographystyle{amsplain}

\title{The Neumann Problem and Helmholtz Decomposition \\
in Convex Domains}

\author{Jun Geng and Zhongwei Shen\footnote{Supported in part by NSF grant DMS-0855294.}}

\date{ }
\maketitle
\begin{abstract}

We show that the Neumann problem for Laplace's equation
in a convex domain $\Omega$ with boundary data
in $L^p(\partial\Omega)$ is uniquely solvable for $1<p<\infty$.
As a consequence, we obtain the Helmholtz decomposition of vector fields
in $L^p(\Omega, \mathbb{R}^d)$.

\end{abstract}

\section{Introduction}

The main purpose of this paper is to prove the following.

\begin{thm}\label{Main-Theorem}
Let $\Omega$ be a bounded convex domain in $\mathbb{R}^d$, $d\ge 2$.
Let $1<p<\infty$.
Then the $L^p$ Neumann problem for
$\Delta u=0$ in $\Omega$ is uniquely solvable.
That is, given any $f\in L^p(\partial\Omega)$ with mean value zero,
there exists a harmonic function $u$ in $\Omega$, unique up to constants,
such that $(\nabla u)^*\in L^p(\partial\Omega)$
and $\frac{\partial u}{\partial n} =f$ n.t. on
$\partial\Omega$.
Moreover, the solution satisfies the estimate
$\|(\nabla u)^*\|_p \le C\| f\|_p$, where
$C$ depends only on $d$, $p$ and the Lipschitz character of $\Omega$.
\end{thm}

Here and thereafter $(\nabla u)^*$ denotes the nontangential maximal function of
$\nabla u$ and $n$ the unit outward normal to $\partial\Omega$.
By $\frac{\partial u}{\partial n}=f$ n.t. on $\partial\Omega$ we mean that
for a.e. $P\in \partial\Omega$,
$<\nabla u(x), n(P)>$ converges to $f(P)$ as $x\to P$ nontangentially.
We remark that for a harmonic function $u$ and $p\ge 2$, 
$(\nabla u)^*\in L^p(\partial\Omega)$ implies that $u$ is in the Sobolev space
$L^p_{1+\frac{1}{p}} (\Omega)$ \cite{JK-1995}.
The solutions in Theorem \ref{Main-Theorem}
satisfy the estimate 
$
 \|\nabla u\|_{L^p_{1/p}(\Omega)}
\le C \| f\|_p
$
for $2\le p<\infty$.

It is known that the $L^p$ Neumann problem 
for $\Delta u=0$ in a bounded $C^1$ domain
is uniquely solvable for any $p\in (1,\infty)$ \cite{FJR-1978}.
However, if $\Omega$ is a general Lipschitz domain, the sharp range of
$p$'s, for which the $L^p$ Neumann problem in $\Omega$ is solvable, is
$1<p<2+\varep$, where $\varep>0$ depends on
$\Omega$ (see \cite{JK-1981, Verchota-1984, Dahlberg-Kenig-1987};
also see \cite{Kenig-1994}
for references on related work on boundary value problems in Lipschitz domains). 
In \cite{Kim-Shen}, for any given Lipschitz domain $\Omega$ and $p>2$,
 Kim and Shen established a necessary and sufficient condition
for the solvability of the $L^p$ Neumann problem in $\Omega$.
More precisely, it is shown in \cite{Kim-Shen}
that the $L^p$ Neumann problem for $\Delta u=0$ in $\Omega$ is solvable
if and only if there exist positive constants $C_0$ and $r_0$ such that
for any $0<r<r_0$ and $Q\in \partial\Omega$, the
following weak reverse H\"older inequality on $\partial\Omega$,
\begin{equation}\label{reverse-Holder}
\left\{ \frac{1}{r^{d-1}}
\int_{B(Q,r)\cap \partial\Omega} |(\nabla v)^*|^p\, d\sigma\right\}^{1/p}
\le C_0 
\left\{ \frac{1}{r^{d-1}}
\int_{B(Q,2r)\cap\partial\Omega} |(\nabla v)^*|^2\, d\sigma\right\}^{1/2},
\end{equation}
holds for any harmonic function $v$ in $\Omega$
satisfying $(\nabla v)^*\in L^2(\partial\Omega)$
and $\frac{\partial v}{\partial n}=0$
 on $B(Q,3r)\cap \partial\Omega$ (see \cite[Theorem 1.1]{Kim-Shen}).
Using this condition,
Kim and Shen \cite{Kim-Shen} obtained the solvability of the $L^p$ Neumann problem
for $\Delta u=0$ in bounded convex domains in $\mathbb{R}^d$
for $1<p<\infty$ if $d=2$; for $1<p<4$ if $d=3$; and 
for $1<p<3+\varep$ if $d\ge 4$ (see \cite[Theorem 1.2]{Kim-Shen}).
Theorem \ref{Main-Theorem} extends the results in \cite{Kim-Shen}
in the case $d\ge 3$
and completely solves the $L^p$ Neumann problem for
Laplace's equation in convex domains.

Our approach to Theorem \ref{Main-Theorem} follows the proof of Theorem 1.2
in \cite{Kim-Shen}.
To establish the weak reverse H\"older inequality (\ref{reverse-Holder}),
we use the square function estimates for harmonic functions in Lipschitz domains
and the local $W^{2,2}$ estimate in convex domains.
This reduces the problem to the estimate of
\begin{equation}\label{key-term}
\sup_{x\in B(P,r)} |\nabla^2 v(x)|^{p-2} [\delta (x)]^{p-1-t},
\end{equation} 
where $t\in (0,1)$, $\delta(x)=\text{dist}(x,\partial\Omega)$,
and $v$ is a harmonic function in $\Omega$ such that $\frac{\partial v}{\partial n}=0$
on $B(P,3r)\cap \partial\Omega$ and $(\nabla v)^*\in L^2(\partial\Omega)$.
 In \cite{Kim-Shen} the authors used the interior estimates and
the local $W^{2,2}$ to obtain that for any $ x\in B(P,r)\cap \Omega$,
\begin{equation}\label{lower-dim}
|\nabla^2 v(x)| \le \frac{C}{r} \left[ \frac{r}{\delta (x)}\right]^{\frac{d}{2}}
\left\{ \frac{1}{r^d} \int_{B(P,3r)\cap \Omega}
|\nabla v(y)|^2\, dy\right\}^{1/2}.
\end{equation}
By a reflection argument the classical De Giorgi-Nash
estimate implies that for any $x\in B(P,r)\cap \Omega$,
\begin{equation}\label{high-dim}
|\nabla^2 v(x)| \le \frac{C}{r} \left[ \frac{r}{\delta (x)}\right]^{2-\alpha}
\left\{ \frac{1}{r^d} \int_{B(P,3r)\cap \Omega}
|\nabla v(y)|^2\, dy\right\}^{1/2},
\end{equation}
where $\alpha>0$ depends on $\Omega$.
Substituting (\ref{lower-dim}) for $d=2,3$ and  (\ref{high-dim}) for $d\ge 4$ into
(\ref{key-term}) and choose $t$ sufficiently close to $0$, we see that
the exponent of $\delta (x)$ would be positive for any $p>2$ if $d=2$; 
for $p<4$ if $d=3$; and for $p<3+\varep$
if $d\ge 4$. 
This leads to the restriction of $p$ for $d\ge 3$ in \cite[Theorem 1.2]{Kim-Shen}.
In this paper we will show that if $d\ge 3$, for any $x\in B(P,r)\cap \Omega$,
\begin{equation}\label{new-high-dim}
|\nabla^2 v(x)| \le \frac{C_\eta }{r} \left[ \frac{r}{\delta (x)}\right]^{1+\eta}
\left\{ \frac{1}{r^d} \int_{B(P,3r)\cap \Omega}
|\nabla v(y)|^2\, dy\right\}^{1/2}
\end{equation}
for any $\eta>0$. Substituting (\ref{new-high-dim}) into (\ref{key-term}), we see that
the exponent of $\delta(x)$ is
$-\eta (p-2) +1-t$, which would be positive for any $p>2$ if $\eta>0$ is sufficiently small.

To show (\ref{new-high-dim}), we will prove that if $\Omega$ is a convex domain
with smooth boundary,
then for any $q>2$,
\begin{equation}\label{reverse-1}
\left\{ \frac{1}{r^d}
\int_{B(Q,r)\cap \Omega}
|\nabla v|^q\, dx\right\}^{1/q}
\le C 
\left\{ \frac{1}{r^d}
\int_{B(Q,2r)\cap \Omega}
|\nabla v|^2\, dx\right\}^{1/2},
\end{equation}
whenever $v$ is harmonic in $\Omega$ and
$v\in C^2(\overline{\Omega})$, $\frac{\partial v}{\partial n}
=0$ in $B(Q,3r)\cap \partial\Omega$.
The constant $C$ in (\ref{reverse-1}) depends only on
$d$, $q$ and the Lipschitz character of $\Omega$. 
Our proof of (\ref{reverse-1}) is 
inspired by a recent paper of V. Maz'ya \cite{Mazya-2008}, in which he
established the $L^\infty$ gradient estimate for
solutions of the Neumann-Laplace problem in convex domains.
More precisely, it is proved in \cite{Mazya-2008} that if $q>d$ and $f\in L^q(\Omega)$
with mean value zero, then
$
\|\nabla u\|_{L^\infty(\Omega)} \le C \| f\|_{L^q(\Omega)},
$
where $-\Delta u=f$ in $\Omega$ and $\frac{\partial u}{\partial n}=0$
on $\partial\Omega$.
Although the proof of (\ref{reverse-1}) does not rely on this estimate,
the formulation of our main technical lemma, Lemma \ref{lemma-2.1}, as well as its proof,
is motivated by \cite{Mazya-2008}.

With Theorem \ref{Main-Theorem} at our disposal, following the potential approach
developed by Fabes, Mendez, Mitrea \cite{FMM-1998}
in Lipschitz domains, we may study the solvability of the Poisson
equation with Neumann boundary conditions in convex domains.
In particular, consider the boundary value problem
\begin{equation}\label{Poisson-problem}
\left\{
\aligned 
& \Delta u=f \in L^p_{-1,0}(\Omega),\\
&\frac{\partial u}{\partial n}=g \in B^p_{-1/p}(\partial\Omega),\\
& u\in W^{1,p}(\Omega).
\endaligned
\right.
\end{equation}
Here $L^p_{-1,0}(\Omega)$ is the dual of
$L^q_1(\Omega)=W^{1,q}(\Omega)$ and $B^p_{-1/p}(\partial\Omega)$ the dual of 
the Besov space $B^q_{1/p}(\partial\Omega)$ 
on $\partial\Omega$, where $q=\frac{p}{p-1}$.
 We will call $u\in W^{1,p}(\Omega)$ a solution to (\ref{Poisson-problem})
with data $(f,g)$, if 
\begin{equation}\label{weak-solution}
\int_\Omega \nabla u \cdot\nabla\phi\, dx
=-<f, \phi>_{L^p_{-1,0}(\Omega)\times L^q_1(\Omega)}
 +<g, Tr(\phi)>_{B^p_{-1/p}(\partial\Omega)\times B^q_{1/p}(\partial\Omega)}
\end{equation}
for any $\phi\in W^{1,q}(\Omega)$, where $Tr(\phi)$ denotes the trace of $\phi$ on $\partial\Omega$.

\begin{thm}\label{cor-1}
Let $\Omega$ be a bounded convex domain in $\mathbb{R}^d$, $d\ge 2$.
Let $1<p<\infty$.
Then for any $f\in L^p_{-1,0}(\Omega)$ and $g\in B^p_{-1/p}(\partial\Omega)$ satisfying the
compatibility condition $<f,1>=<g,1>$, the Poisson problem (\ref{Poisson-problem})
has a unique (up to constants) solution $u$.
Moreover, the solution $u$ satisfies the estimate
\begin{equation}
\|\nabla u\|_{L^p(\Omega)}
\le C \left\{ \| f\|_{L^p_{-1,0}(\Omega)}
+\| g\|_{B^p_{-1/p}(\partial\Omega)}\right\},
\end{equation}
where $C$ depends only on $d$, $p$ and the Lipschitz character of $\Omega$.
\end{thm}

We remark that for bounded Lipschitz or $C^1$ domains, the inhomogeneous Neumann problem,
$\Delta u=f\in L^p_{1/p-s-1,0}(\Omega)$ in $\Omega$, $\frac{\partial u}{\partial n}
=g\in B_{-s}^p(\partial\Omega)$ and $u\in L^p_{1-s+1/p}(\Omega)$
with $s\in (0,1)$ and $p\in (1, \infty)$,
 was studied in \cite{FMM-1998}, where the authors obtained
the solvability for the sharp ranges of
$p$ and $s$. Analogous results 
for the inhomogeneous Dirichlet problem in Lipschitz or $C^1$ domains
may be found in \cite{JK-1995}.
In particular, it follows from \cite{FMM-1998} that
the boundary value problem (\ref{Poisson-problem}) is solvable for
$p\in ((3/2)-\varep, 3+\varep)$ if $\Omega$ is Lipschitz; and for $p\in (1, \infty)$
if $\Omega$ is $C^1$.

Let $L^p_\sigma (\Omega)$ denote the subspace of
functions $\mathbf{v}$ in $L^p (\Omega, \mathbb{R}^d)$ such that
$
\int_\Omega \mathbf{v}\cdot \nabla \phi dx=0$
for any $\phi\in C^1(\mathbb{R}^d)$.
As a corollary of Theorem \ref{cor-1}, we establish the Helmholtz decomposition
of $L^p$ vector fields on convex domains for $1<p<\infty$.

\begin{thm}\label{cor-2}
Let $\Omega$ be a bounded convex domain in $\mathbb{R}^d$, $d\ge 2$ and  $1<p<\infty$. Then
\begin{equation}\label{H-decomposition}
L^p(\Omega, \mathbb{R}^d) =\text{grad } W^{1,p}(\Omega) \oplus L^p_\sigma (\Omega).
\end{equation}
That is, given any $\mathbf{u}\in L^p(\Omega, \mathbb{R}^d)$, there exist 
$\phi \in W^{1,p}(\Omega)$, unique up to a constant, and a unique
$\mathbf{v}\in L^p_\sigma (\Omega)$
such that 
$
\mathbf{u}=\nabla \phi + \mathbf{v}.
$
Moreover, 
\begin{equation}\label{H-estimate}
\max \big\{ \| \nabla \phi\|_{L^p(\Omega)},
\| \mathbf{v}\|_{L^p(\Omega)} \big\} 
\le C_p \, \| \mathbf{u}\|_{L^p(\Omega)},
\end{equation}
where $C_p$ depends only on $d$, $p$ and the Lipschitz character of $\Omega$.
\end{thm}

A useful tool in the study of the Navier-Stokes
equations, 
the Helmholtz decomposition (\ref{H-decomposition})
is well known for smooth domains (see e.g. \cite{FuMo}).
It was proved in \cite{FMM-1998} that
(\ref{H-decomposition})-(\ref{H-estimate}) hold for $p\in ((3/2)-\varep, 3+\varep)$
if $\Omega$ is Lipschitz; and for $p\in (1,\infty)$ if $\Omega$ is $C^1$.
The range $(3/2)-\varep<p<3 +\varep$ is known to be sharp for Lipschitz domains
\cite{FMM-1998}.

\section{Estimates on smooth convex domains}

The purpose of this section is to establish the following.

\begin{thm}\label{Theorem-2.1}
Let $\Omega$ be a bounded convex domain in $\mathbb{R}^d$, $d\ge 3$
with $C^2$ boundary.
Let $u\in C^3(\overline{\Omega})$.
Suppose that $\Delta u=0$ in $\Omega$ and $\frac{\partial u}{\partial n}=0$
on $B(Q,3r)\cap \partial\Omega$ for some $Q\in \partial\Omega$ and $0<r<r_0$.
Then for any $q>2$,
\begin{equation}\label{estimate-2.1}
\left\{ \frac{1}{r^{d}}
\int_{B(Q,r)\cap \Omega}
|\nabla u|^q\, dx\right\}^{1/q}
\le C
\left\{ \frac{1}{r^{d}}
\int_{B(Q,2r)\cap \Omega}
|\nabla u|^2\, dx\right\}^{1/2},
\end{equation}
where $C$ depends only on $d$, $q$ and the Lipschitz character of $\Omega$.
\end{thm}

The summation convention will be used in this section.

The proof of Theorem \ref{Theorem-2.1} relies on the following lemma.
As we mentioned in Introduction, the formulation of Lemma \ref{lemma-2.1}
as well as its proof is inspired
by a paper of Maz'ya \cite{Mazya-2008}.

\begin{lemma}\label{lemma-2.1}
Let $\Omega$ be a bounded convex domain with $C^2$ boundary.
Suppose that $\mathbf{v}=(v_1,\dots, v_d) \in C^2(\overline{\Omega}, \mathbb{R}^d)$ 
and $\mathbf{v}\cdot n =0$ on $\partial \Omega$. 
Let $g=|\mathbf{v}|^2$. Then 
for a.e. $t\in (0,\infty)$,
\begin{equation}\label{estimate-2.1.1}
\aligned
\int_{\{ g=t\}} |\nabla g |\, d\sigma & 
\le  2\sqrt{t}\int_{\{ g=t\}}\left\{ 
\left(\sum_{i,j} \big|\frac{\partial v_i}{\partial x_j}- \frac{\partial v_j}
{\partial x_i}\big|^2\right)^{1/2}+|\text{\rm div} (\mathbf{v})|\right\}d\sigma
\\
&\qquad\qquad
+2 \int_{\{g> t\}}
\left\{ |\text{\rm div}(\mathbf{v})|^2
- \frac{\partial v_i}{\partial x_j} \cdot \frac{\partial v_j}{\partial x_i}\right\} \, dx,
\endaligned
\end{equation}
where $\sigma=H^{d-1}$ denotes the $d-1$ dimensional Hausdorff measure and
$\{ g=t\} =\{ x\in \Omega: g(x)=t\}$, $\{ g>t\} =\{ x\in \Omega: g(x)>t\}$.

\end{lemma}

\begin{proof}
 Let $\Psi$ be a nonnegative Lipschitz function on $[0, \infty)$.
It follows from integration by parts that

\begin{equation}
\aligned
\int_\Omega \Psi(|\mathbf{v}|^2)  
\frac{\partial v_i}{\partial x_j} \cdot \frac{\partial v_j}{\partial x_i}\, dx
& =
-2\int_\Omega \Psi^\prime (|\mathbf{v}|^2)  v_k\cdot \frac{\partial v_k}{\partial x_j}
\cdot v_i \cdot\frac{\partial v_j}{\partial x_i}\, dx\\
 &\qquad\qquad -\int_\Omega \Psi (|\mathbf{v}|^2) \cdot v_i \cdot \frac{\partial }{\partial x_i}
\big\{ \text{\rm div} (\mathbf{v})\big\}\, dx\\
&\qquad\qquad +\int_{\partial\Omega} \Psi (|\mathbf{v}|^2) v_i n_j
\frac{\partial v_j}{\partial x_i}\, d\sigma\\
&=-2\int_\Omega \Psi^\prime (|\mathbf{v}|^2)  v_k\cdot \frac{\partial v_k}{\partial x_j}
\cdot v_i \cdot\frac{\partial v_j}{\partial x_i}\, dx\\
&\qquad \qquad + 2\int_\Omega \Psi^\prime (|\mathbf{v}|^2)
v_k \cdot\frac{\partial v_k}{\partial x_i} \cdot v_i \cdot \text{\rm div}(\mathbf{v})\, dx\\
&\qquad\qquad +\int_\Omega \Psi (|\mathbf{v}|^2) \big\{
\text{\rm div}(\mathbf{v})\big\}^2\, dx\\
& \qquad\qquad +\int_{\partial\Omega}
\Psi (|\mathbf{v}|^2)
\left\{ v_i n_j \frac{\partial v_j}{\partial x_i}
-v_i n_i \text{\rm div} (\mathbf{v})\right\}\, d\sigma.
\endaligned
\end{equation}
This gives
\begin{equation}
\aligned
\int_\Omega & \Psi (|\mathbf{v}|^2)   \left\{
\big\{ \text{\rm div}(\mathbf{v})\big\}^2
-\frac{\partial v_i}{\partial x_j} \cdot \frac{\partial v_j}{\partial x_i}\right\}
\, dx\\
& =\int_{\partial\Omega}
\Psi (|\mathbf{v}|^2)
\left\{ v_i n_i \text{\rm div} (\mathbf{v})
-v_i n_j \frac{\partial v_j}{\partial x_i}
\right\}\, d\sigma
\\
&\qquad
+2\int_\Omega
\Psi^\prime (|\mathbf{v}|^2)
\left\{
v_k \cdot \frac{\partial v_k}{\partial x_j} \cdot v_i \cdot \frac{\partial v_j}{\partial x_i}
-v_k \cdot \frac{\partial v_k}{\partial x_i}
\cdot v_i \cdot \text{\rm div}(\mathbf{v})\right\}\, dx.
\endaligned
\end{equation}

Using the assumptions that $\mathbf{v}\cdot n=0$ on $\partial\Omega$ and
$\Omega$ is a convex  domain with $C^2$ boundary, we observe that
$$
v_i n_i \text{\rm div} (\mathbf{v})
-v_i n_j \frac{\partial v_j}{\partial x_i}
=-\beta (\mathbf{v}_T; \mathbf{v}_T)
\ge 0 \qquad \text{ on }\partial\Omega,
$$
where $\mathbf{v}_T=\mathbf{v}-  (\mathbf{v}\cdot n) n$ is the tangential component
of $\mathbf{v}$ on $\partial\Omega$
and $\beta (\cdot, \cdot)$ the second fundamental quadratic form of $\partial\Omega$
(see \cite[p.137]{Grisvard}). Hence,
\begin{equation}\label{2.3}
\aligned
2\int_\Omega 
\Psi^\prime (|\mathbf{v}|^2) &
\cdot v_k \cdot \frac{\partial v_k}{\partial x_j} \cdot v_i \cdot \frac{\partial v_j}{\partial x_i}
\, dx\\
&\le 2 \int_\Omega 
\Psi^\prime (|\mathbf{v}|^2)
\cdot v_k \cdot \frac{\partial v_k}{\partial x_i}
\cdot v_i \cdot \text{\rm div}(\mathbf{v})\, dx\\
&\qquad\qquad
+
\int_\Omega  \Psi (|\mathbf{v}|^2)   \left\{
\big\{ \text{\rm div}(\mathbf{v})\big\}^2
-\frac{\partial v_i}{\partial x_j} \cdot \frac{\partial v_j}{\partial x_i}\right\}
\, dx.
\endaligned
\end{equation}
Let $g=|\mathbf{v}|^2$. Then $|\nabla g|^2 =4 v_k \cdot \frac{\partial v_k}{\partial x_j}
\cdot v_i \cdot \frac{\partial v_i}{\partial x_j}$. 
It follows  from (\ref{2.3}) that
\begin{equation}\label{2.4}
\aligned
\frac12 \int_\Omega \Psi^\prime (g) \big|\nabla g |^2\, dx
&\le
2\int_\Omega
\Psi^\prime (g ) v_k\cdot \frac{\partial v_k}{\partial x_j}
\cdot v_i \left\{ \frac{\partial v_i}{\partial x_j}
-\frac{\partial v_j}{\partial x_i}\right\}\, dx\\
 &  \qquad \qquad +2 \int_\Omega 
\Psi^\prime (g )
\cdot v_k \cdot \frac{\partial v_k}{\partial x_i}
\cdot v_i \cdot \text{\rm div}(\mathbf{v})\, dx\\
&\qquad\qquad
+
\int_\Omega  \Psi (g )   \left\{
\big\{ \text{\rm div}(\mathbf{v})\big\}^2
-\frac{\partial v_i}{\partial x_j} \cdot \frac{\partial v_j}{\partial x_i}\right\}
\, dx.
\endaligned
\end{equation}

We now fix $0<t<\tau<\infty$. Let $\Psi$ be continuous so that
$\Psi(s)=1$ for $s\ge \tau $, $\Psi (s)=0$ for $s\le t $, and $\Psi$ is linear on $[t, \tau]$.
In view of (\ref{2.4}), we obtain
\begin{equation}\label{2.5}
\aligned
\frac{1}{2(\tau-t)}
\int_{t< g < \tau} |\nabla g|^2\, dx
& \le \frac{1}{\tau-t}
\int_{t< g<\tau}
|\nabla g|\, |\mathbf{v}| \left\{ \sum_{i,j} \big|
\frac{\partial v_i}{\partial x_j} -\frac{\partial v_j}{\partial x_i}\big|^2\right\}^{1/2}
\, dx
\\
& 
+ \frac{1}{\tau-t} \int_{t< g<\tau} 
|\nabla g|\, |\mathbf{v}|\, |\text{\rm div} (\mathbf{v})|\, dx\\
&
+\int_{g> t}
\Psi(g) \left\{  \big\{ \text{\rm div}(\mathbf{v})\big\}^2
-\frac{\partial v_i}{\partial x_j}
\cdot \frac{\partial v_j}{\partial x_i}\right\} \, dx.
\endaligned
\end{equation}
By the co-area formula, we may rewrite (\ref{2.5}) as
\begin{equation}\label{2.6}
\aligned
\frac{1}{2(\tau-t)}
\int_t^\tau \int_{g=s}  |\nabla g|\, d\sigma \, ds
\le  & \frac{1}{\tau-t} \int_t^\tau \int_{g=s}
|\mathbf{v}|
\left\{ \sum_{i,j} \big|
\frac{\partial v_i}{\partial x_j} -\frac{\partial v_j}{\partial x_i}\big|^2\right\}^{1/2}
\, d\sigma\, ds\\
& +\frac{1}{\tau -t}
\int_t^\tau \int_{g=s}
|\mathbf{v}|\, |\text{\rm div} (\mathbf{v})|\, d\sigma\, ds\\
&
+\int_{g> t}
\Psi(g) \left\{  \big\{ \text{\rm div}(\mathbf{v})\big\}^2
-\frac{\partial v_i}{\partial x_j}
\cdot \frac{\partial v_j}{\partial x_i}\right\} \, dx.
\endaligned
\end{equation}
Letting $\tau\to t^+$ in (\ref{2.6}), we obtain
the desired estimate by the Lebesgue's differentiation theorem.
\end{proof}

Next we apply Lemma \ref{lemma-2.1} to harmonic functions in $\Omega$ with normal 
derivatives vanishing on part of the boundary.

\begin{lemma}\label{lemma-2.2}
Let $\Omega$ be a bounded convex domain with $C^2$ boundary and $Q\in \partial\Omega$.
Let $u\in C^3(\overline{\Omega})$.
Suppose that $\Delta u=0$ in $\Omega$
and $\frac{\partial u}{\partial n}=0$ on $B(Q,2r)\cap \partial\Omega$ for some $r>0$.
Then for a.e. $t\in (0,\infty)$,
\begin{equation}\label{estimate-2.2.1}
\int_{g=t} |\nabla g|\, d\sigma
\le 6\sqrt{t} \int_{g=t} |\nabla u|\, |\nabla \varphi|\, d\sigma
+2\int_{g> t} |\nabla u|^2 |\nabla \varphi|^2\, dx,
\end{equation}
where $g=|(\nabla u)\varphi|^2$ and $\varphi \in C_0^\infty(B(Q,2r))$.
\end{lemma}

\begin{proof}
Let $\mathbf{v}=(\nabla u)\varphi$. Then $\mathbf{v}\cdot n=0$ on $\partial\Omega$
and
$$
\frac{\partial v_i}{\partial x_j}
=\varphi \frac{\partial^2 u}{\partial x_i\partial x_j}
+\frac{\partial u}{\partial x_i}\frac{\partial \varphi}{\partial x_j}.
$$
It follows that $\text{\rm div}(\mathbf{v})
=(\Delta u)\varphi +\nabla u\cdot \nabla \varphi
=\nabla u\cdot \nabla \varphi$ and
$$
\frac{\partial v_i}{\partial x_j}
-
\frac{\partial v_j}{\partial x_i}
=\frac{\partial u}{\partial x_i}
\frac{\partial \varphi}{\partial x_j}
-
\frac{\partial u}{\partial x_j}
\frac{\partial \varphi}{\partial x_i}.
$$
Hence,
\begin{equation}\label{2.2.1}
\aligned
 &\left(\sum_{i,j}
\big|\frac{\partial v_i}{\partial x_j}
-
\frac{\partial v_j}{\partial x_i}\big|^2\right)^{1/2}
+|\text{\rm div}(\mathbf{v})|\\
&\qquad =\big\{
 2|\nabla u|^2|\nabla \varphi|^2 -2(\nabla u\cdot\nabla \varphi)^2 \big\}^{1/2}
+|\nabla u\cdot \nabla \varphi|\\
&\qquad \le 3|\nabla u|\, |\nabla\varphi|.
\endaligned
\end{equation}

Next note that
\begin{equation}\label{2.2.2}
\aligned
|\text{\rm div}(\mathbf{v})|^2 -\frac{\partial v_j}{\partial x_i}
\frac{\partial v_i}{\partial x_j}
& =
-\varphi^2 |\nabla^2 u|^2 -2\varphi \cdot \frac{\partial^2 u}{\partial x_i\partial x_j}
\cdot \frac{\partial u}{\partial x_i} \cdot \frac{\partial\varphi}{\partial x_j}\\
&
=-\sum_{i,j}
\left( \varphi \frac{\partial^2 u}{\partial x_i\partial x_j}
-\frac{\partial u}{\partial x_i}
\frac{\partial\varphi}{\partial x_j}\right)^2
+|\nabla u|^2 |\nabla \varphi|^2\\
& \le |\nabla u|^2 |\nabla \varphi|^2.
\endaligned
\end{equation}
In view of (\ref{2.2.1})-(\ref{2.2.2}),
 estimate (\ref{estimate-2.2.1}) in Lemma \ref{lemma-2.2}
 now follows readily from Lemma \ref{lemma-2.1}.
\end{proof}

\begin{lemma}\label{lemma-2.3}
Let $\Omega$ be a bounded convex domain in $\mathbb{R}^d$, $d\ge 3$.
Let $f$, $g$ be two nonnegative functions on $\overline{\Omega}$.
Suppose that $f\in C(\overline{\Omega})$, $g\in C^1(\overline{\Omega})$ and
\begin{equation}\label{estimate-2.3.1}
\int_{g=t}
|\nabla g|\, d\sigma
\le C_0 \left\{ \sqrt{t}
\int_{g=t} |f|\, d\sigma
+\int_{g> t} | f|^2\, dx\right\}
\end{equation}
for a.e. $t\in (0, \infty)$.
Then there exists $C$ depending only on $d$, $q$, $C_0$ and 
the Lipschitz character of $\Omega$ such that
\begin{equation}\label{estimate-2.3.2}
\left\{ \int_\Omega |g|^q\, dx \right\}^{1/q}
\le C \left\{ \int_\Omega |f|^{2p}\, dx \right\}^{1/p}
+C|\Omega|^{\frac{1}{q}-1}
 \int_\Omega |g|\, dx,
\end{equation}
where $p> 1$ and $\frac{1}{q}=\frac{1}{p}-\frac{2}{d}$.

\end{lemma}

\begin{proof} 
By considering $g_\delta =g +\delta$ and then letting $\delta\to 0^+$,
 we may assume that $g$ is bounded from below by a positive constant.
Using the co-area formula and (\ref{estimate-2.3.1}), we obtain
\begin{equation}\label{estimate-2.3.3}
\aligned
\int_\Omega |g|^\alpha  |\nabla g|^2\, dx
& =\int_0^\infty  t^\alpha \int_{g=t} |\nabla g|\, d\sigma dt\\
& \le C_0\int_0^\infty
t^{\alpha}
\left\{
t^{\frac12}
\int_{g=t} |f|\, d\sigma
+\int_{g> t} | f|^2\, dx\right\} dt\\
& \le C
\int_\Omega |g|^{\alpha +\frac12} |\nabla g||f|\, dx
+C\int_\Omega |g|^{\alpha +1} |f|^2\, dx,
\endaligned
\end{equation}
where $\alpha >-1 $. By the Cauchy inequality with an $\varep>0$,
$$
 \int_\Omega |g|^{\alpha +\frac12} |\nabla g||f|\, dx
\le \varep \int_\Omega |g|^\alpha |\nabla g|^2\, dx
+C_\varep \int_\Omega | g|^{\alpha +1} |f|^2\, dx.
$$
This, together with (\ref{estimate-2.3.3}), implies that
\begin{equation}\label{estimate-2.3.4}
\int_\Omega |g|^\alpha  |\nabla g|^2\, dx
\le 
C\int_\Omega |g|^{\alpha +1} |f|^2\, dx.
\end{equation}

Since $\Omega$ is convex, there exists a constant $C$,
depending only on $d$ and $[\text{diam}(\Omega)]^d/|\Omega|$,
such that
\begin{equation}\label{Sobolev-Poincare-inequality}
\left\{ \int_\Omega |w-w_\Omega|^{\frac{2d}{d-2}} \, dx\right\}^{\frac{d-2}{d}}
\le C\int_\Omega |\nabla w|^2\, dx,
\end{equation}
where $w\in C^1(\overline{\Omega})$
and $w_\Omega$ denotes the average of $w$ over $\Omega$.
Let $\beta> (1/2) $ and $w=g^\beta$ in (\ref{Sobolev-Poincare-inequality}). We obtain
\begin{equation}\label{Sobolev-inequality}
\left\{ \int_\Omega
|g|^{\frac{2d\beta}{d-2}}\, dx
\right\}^{\frac{d-2}{d}}
\le C\int_\Omega |g|^{2\beta-2}|\nabla g|^2\, dx
+C|\Omega|^{-1-\frac{2}{d}}
\left\{ \int_\Omega |g|^{\beta}\, dx\right\}^2.
\end{equation}
Let $\alpha=2\beta-2$. It follows from (\ref{estimate-2.3.4})
and (\ref{Sobolev-inequality}) that
\begin{equation}\label{estimate-2.3.5}
\left\{ \int_\Omega
|g|^{\frac{2d\beta}{d-2}}\, dx
\right\}^{\frac{d-2}{d}}
\le 
C\int_\Omega |g|^{2\beta -1} |f|^2\, dx
+C|\Omega|^{-1-\frac{2}{d}}
\left\{ \int_\Omega |g|^{\beta}\, dx\right\}^2,
\end{equation}
for any $\beta> (1/2)$.

We now choose $p_0> 1$ so that $(2\beta-1)p_0 =\frac{2d\beta}{d-2}$.
By H\"older's inequality,
\begin{equation}\label{estimate-2.3.6}
\aligned
\int_\Omega |g|^{2\beta -1} |f|^2\, dx
&\le \left\{\int_\Omega
|g|^{(2\beta-1)p_0}\, dx\right\}^{1/p_0}
\left\{ \int_\Omega |f|^{2p_0^\prime}\, dx \right\}^{1/p_0^\prime}\\
&\le
\varep \left\{ \int_\Omega |g|^{(2\beta-1)p_0}\, dx\right\}^{\frac{p_1}{p_0}}
+C_\varep
\left\{ \int_\Omega |f|^{2 p_0^\prime}\, dx
\right\}^{\frac{p_1^\prime}{p_0^\prime}},
\endaligned
\end{equation}
where $p_1=\frac{2\beta}{2\beta-1}$.
Note that $\frac{p_1}{p_0}=\frac{d-2}{d}$. 
Also $2p_0^\prime
=\frac{4d\beta}{d-2+4\beta}$ and $\frac{p_1^\prime}{p_0^\prime}
=\frac{d-2+4\beta}{d}$.
In view of 
(\ref{estimate-2.3.5})-(\ref{estimate-2.3.6}), we obtain
\begin{equation}\label{2.3.7}
\left\{ \int_\Omega
|g|^{\frac{2d\beta}{d-2}}\, dx\right\}^{\frac{d-2}{d}}
\le C\left\{ \int_\Omega |f|^{\frac{4d\beta}{d-2+4\beta}}
\, dx\right\}^{\frac{d-2+4\beta}{d}}
+C|\Omega|^{-1-\frac{2}{d}}
\left\{ \int_\Omega |g|^{\beta}\, dx\right\}^2.
\end{equation}

Finally we let $p=\frac{2d\beta}{d-2+4\beta}$ and $q=\frac{2d\beta}{d-2}$.
It follows from (\ref{2.3.7}) that
\begin{equation}\label{2.3.8}
\left\{ \int_\Omega |g|^q\, dx \right\}^{1/q}
\le C \left\{ \int_\Omega |f|^{2p}\, dx \right\}^{1/p}
+C|\Omega|^{-\frac{1}{2\beta}-\frac{1}{d\beta}}
\left\{ \int_\Omega |g|^{\beta}\, dx\right\}^{1/\beta}.
\end{equation}
Note that $\frac{1}{q}=\frac{1}{p}-\frac{2}{d}$ and $2\beta=(1-\frac{2}{d})q$.
Also $-\frac{1}{2\beta}-\frac{1}{d\beta}
=\frac{1}{q}-\frac{1}{\beta}$. Since $\beta<q$, the desired estimate follows from (\ref{2.3.8})
by H\"older's inequality.
\end{proof}

We are now in a position to give the proof of Theorem \ref{Theorem-2.1}.

\noindent{\bf Proof of Theorem \ref{Theorem-2.1}.}

Let $1<\rho<\tau<2$.
Choose $\varphi\in C_0^\infty(B(Q,\tau r))$ such that $\varphi=1$ in $B(Q,\rho r)$
and $|\nabla \varphi|\le C[(\tau-\rho)r]^{-1}$.
It follows from Lemmas \ref{lemma-2.2} and \ref{lemma-2.3} that
\begin{equation}\label{2.4.1}
\aligned
& \left\{\int_\Omega 
 |(\nabla u)\varphi|^{2q}\, dx\right\}^{1/q}\\
& \le C \left\{ \int_\Omega |\nabla u|^{2p}
|\nabla \varphi|^{2p}\, dx \right\}^{1/p}
+C |\Omega|^{\frac{1}{q}-1}
\int_\Omega |(\nabla u)\varphi|^{2} \, dx,
\endaligned
\end{equation}
where $p> 1$ and $\frac{1}{q}=\frac{1}{p}-\frac{2}{d}$.
This yields that 
\begin{equation}\label{2.4.2}
 \left\{ \frac{1}{r^d}\int_{\Omega\cap B(Q,\rho r)}
|\nabla u|^{2q}\, dx \right\}^{1/(2q)}
\le C 
\left\{ \frac{1}{r^d} \int_{\Omega\cap B(Q,\tau r)}
|\nabla u|^{2p}\, dx\right\}^{1/(2p)}
\end{equation}
for any $p>1 $ and $\frac{1}{q}=\frac{1}{p}-\frac{2}{d}$, where
$1<\rho <\tau<2$. 
By a simple iteration argument, we obtain
\begin{equation}\label{2.4.10}
\left\{ \frac{1}{r^d}\int_{\Omega\cap B(Q,r)}
|\nabla u|^{q}\, dx \right\}^{1/q}
\le C 
\left\{ \frac{1}{r^d} \int_{\Omega\cap B(Q,3r/2)}
|\nabla u|^{p}\, dx\right\}^{1/p},
\end{equation}
for any $2<p<q<\infty$.
Estimate (\ref{estimate-2.1})
follows readily from (\ref{2.4.10}) and the reverse H\"older inequality,
\begin{equation}\label{classical-reverse}
\left\{ \frac{1}{r^d}\int_{\Omega\cap B(Q,3r/2)}
|\nabla u|^{\bar{p}}\, dx \right\}^{1/\bar{p}}
\le C 
\left\{ \frac{1}{r^d} \int_{\Omega\cap B(Q,2r)}
|\nabla u|^2\, dx\right\}^{1/2},
\end{equation}
where $\bar{p}>2$.
We mention that (\ref{classical-reverse})
holds even for solutions of elliptic systems
of divergence form with bounded measurable coefficients.
See e.g. \cite[Chapter V]{Giaquinta} for the interior case.
The boundary case follows from the interior case by a reflection argument.
\qed

\begin{remark}\label{remark-2.2}
{\rm
Let $\Omega$ be a bounded Lipschitz domain in $\mathbb{R}^d$, $d\ge 2$. 
Suppose that 
$\Delta u=0$ in $\Omega$, $(\nabla u)^*\in L^2(\partial\Omega)$ and
$\frac{\partial u}{\partial n}=0$ on $B(Q,3r)\cap \partial\Omega$.
Then the estimate (\ref{estimate-2.1}) holds for $2<q<3+\varep$ if $d\ge 3$;
and for $2<q<4+\varep$ if $d=2$.
To show this, one uses the fact that the $L^2$ Neumann problem
in Lipschitz domains is solvable as well as the observation
that $u$ is $C^\alpha$ in $B(Q,2r)\cap \overline{\Omega}$ for some $\alpha>0$ if $d\ge 3$;
and for some $\alpha>(1/2)$ if $d=2$.
We refer the reader to \cite[pp.188-189]{Shen-2005}, where 
the same estimate
was proved for a Lipschitz domain $\Omega$,
under the Dirichlet condition $u=0$ on $B(Q,3r)\cap\partial\Omega$.
The proof in \cite{Shen-2005} extends easily to the case of the Neumann boundary condition.
We point out that if $d\ge 3$, the $C^\alpha$  ($\alpha>0$) estimate follows from the De Giorgi-Nash estimate
by a reflection argument.
For the case $d=2$, one may use the solvability of the $L^p$ Neumann problem in Lipschitz domains
for some $p=\bar{p}>2$ and the square function estimates to show that
$\nabla u\in L^{\bar{p}}_{1/\bar{p}}(B(Q,2r)\cap \Omega)\subset L^{2\bar{p}} (B(Q,2r)\cap \Omega)$.
By Sobolev imbedding, this implies that $u$ is $C^\alpha$ on $B(Q,2r)\cap \overline{\Omega}$
for some $\alpha>(1/2)$.
If $\Omega$ is $C^1$, the estimate (\ref{estimate-2.1})
holds for any $d\ge 2$ and $q>2$. This follows from the fact that the $L^p$ Neumann problem
in $C^1$ domains is solvable for any $p>2$.
Since the results in this paper do not use the estimates mentioned above, 
we omit the details here.
}
\end{remark}

\section{Weak reserve H\"older inequality on the boundary}

The goal of this section is to prove the following.

\begin{thm}\label{Theorem-3.1}
Under the same conditions on $\Omega$ and $u$ as in Theorem \ref{Theorem-2.1}, we have
\begin{equation}\label{estimate-3.1}
\left\{ \frac{1}{r^{d-1}}
\int_{B(Q,r)\cap \partial\Omega} 
|(\nabla u)^*|^p\, d\sigma \right\}^{1/p}
\le C 
\left\{ \int_{B(Q,2r)\cap \partial\Omega} 
|(\nabla u)^*|^2 \, d\sigma \right\}^{1/2},
\end{equation}
for any $p>2$, where $C$ depends only on $d$, $p$ and the Lipschitz character of $\Omega$.
\end{thm}

We begin with a local $W^{2,2}$ estimate.

\begin{lemma}\label{lemma-3.2}
Under the same conditions on $\Omega$ and $u$ as in Theorem \ref{Theorem-2.1},
we have
\begin{equation}\label{estimate-3.2}
\int_{B(Q,r)\cap \Omega}
|\nabla^2 u|^2\, dx \le \frac{C}{r^2} 
\int_{B(Q,2r)\cap \Omega} |\nabla u|^2\, dx
\end{equation}
where $C$ depends only on $d$.
\end{lemma}

\begin{proof}
See e.g. \cite[p.1826]{Kim-Shen}.
\end{proof}

Let $\delta(x)=\text{dist}(x, \partial\Omega)$.

\begin{lemma}\label{lemma-3.3}
Let $w$ be a harmonic function in a bounded Lipschitz domain $\Omega$.
Let $p>2$. Fix $x_0\in \Omega$ such that $\delta(x_0)\ge c_0 \text{diam}(\Omega)$.
Then for any $t\in (0,1)$,
\begin{equation}\label{estimate-3.3}
\aligned
\int_{\partial\Omega}
|(\nabla w)^*|^p\, d\sigma \le 
&
C \big\{ \text{diam} (\Omega)\big\}^t \sup_{x\in \Omega}
|\nabla^2 w (x)|^{p-2} [\delta (x)]^{p-1-t}
\int_\Omega |\nabla^2 w|^2\, dy\\
&\qquad + C |\nabla w(x_0)|^p |\partial\Omega|,
\endaligned
\end{equation}
where $C$ depends only on
$d$, $p$, $t$, $c_0$ and the Lipschitz character of $\Omega$.
\end{lemma}

\begin{proof} See e.g. \cite[p.1827]{Kim-Shen}.
\end{proof}

\noindent{\bf Proof of Theorem \ref{Theorem-3.1}.}

Since $\Omega$ is a Lipschitz domain, 
by rotation and translation, we may assume that $Q=0$ and
$$
B(Q, C_0 r_0)\cap\Omega
=\big\{ (x^\prime, x_d):\ x_d>\psi (x^\prime)\big\}\cap B(Q, C_0r_0)
$$
where $\psi: \mathbb{R}^{d-1}\to \mathbb{R}$ such that $\psi(0)=0$ and $\|\nabla \psi\|_\infty
\le M$, and $C_0 =10\sqrt{d} (1+M)$.
Let 
$$
S(r)= \big\{ (x^\prime, \psi(x^\prime)): |x^\prime|<r\big \}.
$$
We will show that if $u\in C^2(\overline{\Omega})$ is harmonic in $\Omega$
and $\frac{\partial u}{\partial n}=0$ on $S(8r)$, then
\begin{equation}\label{estimate-3.4}
\left\{ \frac{1}{r^{d-1}}
\int_{S(r)} 
|(\nabla u)^*|^p\, d\sigma \right\}^{1/p}
\le C 
\left\{ \frac{1}{r^{d-1}}\int_{S(4r)} 
|(\nabla u)^*|^2 \, d\sigma \right\}^{1/2},
\end{equation}
where $C$ depends only on $d$, $p$ and $M$.
Estimate (\ref{estimate-3.1}) follows from (\ref{estimate-3.4})
by a simple covering argument.

For $P\in \partial\Omega$, define
\begin{equation}
\aligned
\mathcal{M}_1 (\nabla u) (P) & =\sup\big\{ |\nabla u(x)|: x\in \Omega, 
|x-P|<C_0\delta (x) \text{ and } |x-P|\le c_0 r\big \},\\
\mathcal{M}_2 (\nabla u) (P) & =\sup\big\{ |\nabla u(x)|: x\in \Omega, 
|x-P|<C_0\delta (x) \text{ and } |x-P|> c_0 r\big\}
\endaligned
\end{equation}
Note that $(\nabla u)^*=\max \{ \mathcal{M}_1 (\nabla u), \mathcal{M}_2 (\nabla u)\}$.
The desired estimate for $\mathcal{M}_2 (\nabla u)$ follows readily from the 
interior estimates for harmonic functions.
To handle $\mathcal{M}_1(\nabla u)$, we apply Lemma \ref{lemma-3.3} to $u$ on the Lipschitz
domain $Z(2r)$, where
$$
Z (\rho) (=\big\{ (x^\prime, x_d): |x^\prime|<\rho \text{ and }
\psi(x^\prime)< x_d < 20\sqrt{d} (1+M)\rho \big\}.
$$
This yields that
\begin{equation}\label{estimate-3.5}
\aligned
\frac{1}{r^{d-1}}\int_{S(r)} |\mathcal{M}_1 (\nabla u)|^p\, d\sigma
& \le \frac{1}{r^{d-1}} \int_{\partial Z(2r)} |(\nabla u)^*_{Z(2r)}|^p\, d\sigma\\
& \le C r^{t-d+1} \sup_{Z(2r)} |\nabla^2 u(x)|^{p-2} \big[ \delta (x)\big]^{p-1-t}
\int_{Z(2r)} |\nabla^2 u(y)|^2\,dy\\
&\qquad +C|\nabla v(x_0)|^p,
\endaligned
\end{equation}
where $\delta(x)=\text{dist}(x, Z(2r))$ and
$(\nabla u)^*_{Z(2r)}$ denotes the nontangential maximal function of
$\nabla u$ with respect to the domain $Z(2r)$.
Note that the last term in the right-hand side of (\ref{estimate-3.5})
may be treated easily, using the interior estimates.

Let $I$ denote the first term in the right-hand side of (\ref{estimate-3.5}).
By Lemma \ref{lemma-3.2},
\begin{equation}\label{estimate-3.6}
I\le Cr^{t-d-1}
\sup_{Z(2r)} |\nabla^2 u(x)|^{p-2} \big[ \delta (x)\big]^{p-1-t}
\int_{Z(2r)} |\nabla u(y)|^2\,dy.
\end{equation}
Let $x\in Z(2r)$. It follows from the interior estimates that
for any $q>2$,
\begin{equation}\label{estimate-3.7}
\aligned
|\nabla^2 u(x)| &
\le \frac{C}{\delta (x)}
\left \{ \frac{1}{[\delta (x)]^d} \int_{B(x, \delta(x))} 
|\nabla u|^q \, dx \right\}^{1/q}\\
& \le \frac{C r^{\frac{d}{q}} }{[\delta (x)]^{1+\frac{d}{q}}}
\left\{ \frac{1}{r^d}
\int_{Z(2r)} |\nabla u|^q \, dx \right\}^{1/q}\\
& \le 
\frac{C r^{\frac{d}{q}}}{[\delta (x)]^{1+\frac{d}{q}}}
\left\{ \frac{1}{r^d}
\int_{Z(4r)} |\nabla u|^2 \, dx \right\}^{1/2},
\endaligned
\end{equation}
where we have used estimate (\ref{estimate-2.1}) in the last step. This, together with (\ref{estimate-3.6}),
implies that
\begin{equation}\label{estimate-3.8}
I\le C r^{t-1 +\frac{d}{q}(p-2)}
\sup_{x\in Z(2r)}
\big[ \delta (x)\big]^{p-1-t -(1+\frac{d}{q})(p-2)}
\left\{ \frac{1}{r^d}
\int_{Z(4r)}
|\nabla u|^2\, dy\right\}^{p/2}.
\end{equation}
Since $p-1-t -(1+\frac{d}{q})(p-2)
=1-t -\frac{d}{q}(p-2)$, we may choose $q>2$ so large that 
the exponent of $\delta (x)$ in (\ref{estimate-3.8}) is positive.
Using $\delta (x)\le Cr$, we then obtain
\begin{equation}\label{estimate-3.9}
I\le 
C \left \{ \frac{1}{r^d}
\int_{Z(4r)}
|\nabla u|^2\, dy\right\}^{p/2}
\le C \left\{ \frac{1}{r^{d-1}}
\int_{S(4r)}
|(\nabla u)^*|^2\, d\sigma \right\}^{p/2}.
\end{equation}
Thus we have proved that
$$
\frac{1}{r^{d-1}}
\int_{S(r)}
|\mathcal{M}_1 (\nabla u)|^p\, d\sigma
\le C \left\{ \frac{1}{r^{d-1}}
\int_{S(4r)}
|(\nabla u)^*|^2\, d\sigma \right\}^{p/2}.
$$
This, together with the same estimate for $\mathcal{M}_2(\nabla u)$,
gives (\ref{estimate-3.4}).
\qed

\begin{remark}\label{remark-3.1}
{\rm 
Let $p>2$.
It follows from Theorem \ref{Theorem-3.1} and \cite[Theorem 1.1]{Kim-Shen}
 that if $\Omega$ is a bounded convex domain with $C^2$ boundary,
the $L^p$ Neumann problem for $\Delta u=0$ in $\Omega$ is uniquely solvable.
Moreover, the solution satisfies the estimate
$\|(\nabla u)^*\|_p \le C\| \frac{\partial u}{\partial n}\|_p$,
where $C$ depends only on $d$, $p$ and the Lipschitz character of $\Omega$.
}
\end{remark}

\section{Proof of Theorem \ref{Main-Theorem}}

Let $\Omega$ be a bounded convex domain in $\mathbb{R}^d$, $d\ge 2$. Let $p>2$.
We need to show that the $L^p$ Neumann problem for $\Delta u=0$ in $\Omega$
is uniquely solvable. Since the case $d=2$ is contained in \cite{Kim-Shen},
we will assume that $d\ge 3$.

The uniqueness of the $L^p$ Neumann problem follows directly from the uniqueness
of the $L^2$ Neumann problem.
To establish the existence, it suffices to show that if $f\in C^\infty_0 (\mathbb{R}^d)$
and $\int_{\partial\Omega} f\, d\sigma =0$, then the solution of
the $L^2$ Neumann problem for $\Delta u=0$ in $\Omega$ with boundary data $f|_{\partial\Omega}$
satisfies $\| (\nabla u)^*\|_p \le C \|f\|_p$.

To this end we approximate $\Omega$ from the outside
by a sequence of convex domains $\{ \Omega_j\}$
with smooth boundaries and uniform Lipschitz characters.
Let $u_j$ be a solution to the $L^2$ Neumann problem for Laplace's equation
in $\Omega_j$ with data $f_j -\alpha_j$, where $\alpha_j$ is the mean value of $f$
on $\partial\Omega_j$.
It follows from Remark \ref{remark-3.1} that
\begin{equation}\label{estimate-4.1}
\|(\nabla u_j)^*\|_{L^p(\partial\Omega_j)}
\le C\| f_j -\alpha_j \|_{L^p(\partial\Omega_j)},
\end{equation}
where $C$ depends only on $d$, $p$ and the Lipschitz character of $\Omega$.
By a limiting argument (see e.g. \cite{JK-1981}),
there exists a subsequence, still denoted by $\{ u_j\}$, 
such that $\nabla u_j \to \nabla v$ uniformly
on any compact subset of $\Omega$, where $v$ is a variational solution
of the Neumann problem in $\Omega$ with data $f|_{\partial\Omega}$.
Using this and (\ref{estimate-4.1}), we may deduce that
 $\|(\nabla v)^*\|_{L^p(\partial\Omega)}
\le C\| f\|_{L^p(\partial\Omega)}$.
Since $u-v$ is constant by the uniqueness of the 
variational solutions, we obtain $\|(\nabla u)^*\|_{L^p(\partial\Omega)}
\le C\| f\|_{L^p(\partial\Omega)}$.
This completes the proof.

\section{Proof of Theorem \ref{cor-1}}

To establish the existence, we first reduce the problem to the case where $f=0$.
The argument is standard. Let $f\in L^p_{-1,0}(\Omega)$ and
$w=\Pi_\Omega (f) =:\mathcal{R}_\Omega \Pi (\widetilde{f})$.
Here $\mathcal{R}_\Omega$ denotes the operator restricting distributions
in $\mathbb{R}^d$ to $\Omega$, the map $\Pi: \mathcal{E}^\prime (\mathbb{R}^d) \to
\mathcal{D}^\prime (\mathbb{R}^d)$ is given by the convolution
with the fundamental solution for $\Delta$ in $\mathbb{R}^d$ with pole at the origin,
and $\widetilde{f}$ is defined by $<\widetilde{f}, \phi>
=<f, \mathcal{R}_\Omega (\phi)>$ for $\phi \in C^\infty (\mathbb{R}^d)$.
Let $q=\frac{p}{p-1}$.
For any $\phi \in B_{s}^q (\partial\Omega)$ with $s=\frac{1}{p}$, define 
\begin{equation}
<\Lambda (f), \phi> =\int_\Omega \nabla w\cdot \nabla \psi \, dx
+< f, \psi> _{L^p_{-1,0}(\Omega)\times
L^q_1(\Omega)},
\end{equation}
where
$\psi$ is a function in $L^q_1(\Omega)$ such that $Tr(\psi)=\phi$ and 
$\| \psi\|_{L^q_1(\Omega)} \le C \| \phi\|_{B^q_{1/p}(\partial\Omega)}$.
Here we have used the fact that the trace operator
$Tr: L^q_1(\Omega)\to B^q_{1/p}(\partial\Omega)$
is bounded and onto and that
$$
\|\phi\|_{B^q_{1/p}(\partial\Omega)}
\approx
\inf \big\{
\|\psi\|_{L^q_1(\Omega)}: \
Tr(\psi)=\phi\big\}.
$$
Since
$$
\int_\Omega \nabla w\cdot \nabla \psi\, dx +<f, \psi>=0
\qquad \text{ for any } \psi \in C_0^\infty (\Omega)
$$
and $C_0^\infty(\Omega)$ is dense in $\{ u\in L^q_1(\Omega): Tr(u)=0\}$,
it is easy to see that $<\Lambda (f), \phi>$ is well defined. Furthermore,
\begin{equation}
\aligned
|<\Lambda(f), \phi>| &\le
\big\{ \| w\|_{L^p_1(\Omega)} +\| f\|_{L^p_{-1,0}(\Omega)} \big\} \| \psi \|_{L^q_1(\Omega)}\\
&\le C \| f\|_{L^p_{-1,0}(\Omega)}  \| \psi \|_{L^q_1(\Omega)}\\
&\le C\| f\|_{L^p_{-1,0}(\Omega)} \| \phi\|_{B^q_{1/p}(\partial\Omega)},
\endaligned
\end{equation}
where we have used the Calder\'on-Zygmund
estimate $\|w\|_{L^p_1(\Omega)} \le C\| f\|_{L^p_{-1,0}(\Omega)}$.
It follows that $w$ is a weak solution to (\ref{Poisson-problem}) with data $(f, \Lambda(f))$
and
$\| \Lambda (f)\|_{B^p_{-1/p}(\partial\Omega)} \le C \| f\|_{L^p_{-1,0}(\Omega)}$.
Thus, by subtracting $w$ from $u$, we may always assume that $f=0$.

Next, we note that if $\Omega$ is a bounded Lipschitz domain,
the solvability of the Neumann problem, 
\begin{equation}\label{B-Neumann-problem}
\left\{
\aligned
&\Delta u=0 \text{ in } \Omega,\\
&\frac{\partial u}{\partial n}=g \in B_{-s}^p (\partial\Omega) \text{ on }\partial\Omega,\\
& u\in L^p_{1-s+1/p}(\Omega),
\endaligned
\right.
\end{equation}
was established in \cite{FMM-1998} for $(s, 1/p)$
in the (open) convex polygon $\mathcal{P}$ formed by the vertices, 
$$
(1-\varep, 0),\ 
(1,0), \
(1, (1+\varep)/2),\
(\varep , 1),\
(0,1),\
(0, (1-\varep)/2),
$$ 
where
$\varep>0$ depends on $\Omega$. By interpolation,
this, together with Theorem \ref{Main-Theorem},
implies that the Neumann problem (\ref{B-Neumann-problem})
in a convex domain
is uniquely solvable  if $(s,1/p)$ is the (open) 
convex  polygon $\mathcal{P}_1$ formed by the
vertices
$$
(1-\varep, 0), \
(1,0), \
(1, (1+\varep)/2), \
(\varep, 1),\
(0,1),\
(0,0).
$$
In particular, the Neumann problem (\ref{B-Neumann-problem}) is solvable if $s=1/p$ and $2\le p<\infty$.
As a result, we have proved Theorem \ref{cor-1} for $2\le p<\infty$.
The case $1<p<2$ will be proved by a duality argument, given in the next section
(see Remark \ref{Remark-6.3}).

\section{Proof of Theorem \ref{cor-2}}

Note that
\begin{equation}
\aligned
L^p_\sigma (\Omega)  & =\big\{ \mathbf{u}\in L^p(\Omega, \mathbb{R}^d): \
\int_\Omega \mathbf{u}\cdot \nabla \psi \, dx = 0
\text{ for any } \psi\in W^{1,q}(\Omega)\big\}\\
& =\big\{ \mathbf{u}\in L^p(\Omega, \mathbb{R}^d): \
\text{div}(\mathbf{u})=0 \text{ in }\Omega \text{ and }
\mathbf{u}\cdot n=0 \text{ on } \partial\Omega\big\},
\endaligned
\end{equation}
where $q=\frac{p}{p-1}$. Here $\mathbf{u}\cdot n$ is regarded as an element in $B^p_{-1/p}
(\partial\Omega)$.
Let $X$ be a normed vector space and $S$ a subset of $X$.
The set
$$
S^\perp =\{ \ell \in X^*: \ <\ell, f>=0 \text{ for all } f\in S\, \}
$$
 is called the set of annihilators of $S$.
If $S$ is a closed subspace of
a reflexive Banach space $X$, then $(S^\perp)^\perp=S$
(see e.g. \cite{Schechter}).
With this notation, we may write $L^p_\sigma (\Omega)
=S^\perp\subset X= L^p(\Omega, \mathbb{R}^d)$,
 where $S=\text{grad } W^{1,q}(\Omega)\subset L^q(\Omega, \mathbb{R}^d)$.
Thus the $L^p$-Helmholtz decomposition (\ref{H-decomposition}) may be written as
\begin{equation}\label{H-decomposition-1}
L^p(\Omega, \mathbb{R}^d)
=\text{grad } W^{1,p}(\Omega) \oplus \big\{  \text{grad } W^{1,q}(\Omega)\big\}^\perp.
\end{equation}

\begin{lemma}\label{lemma-6.1}
Let $\Omega$ be a bounded Lipschitz domain in $\mathbb{R}^d$ and $1<p<\infty$.
Then
$C_\sigma^\infty(\Omega)
=\{ \mathbf{v}\in C_0^\infty(\Omega, \mathbb{R}^d): \ \text{\rm div}(\mathbf{v})=0 \}$
is dense in $L^p_\sigma (\Omega)$.
\end{lemma}

\begin{proof}
Let $X^p(\Omega)$ denote the closure of $C_\sigma^\infty(\Omega)$
in $L^p(\Omega, \mathbb{R}^d)$.
Clearly, $\text{grad } W^{1,q}(\Omega)\subset (X^p(\Omega))^\perp$.
On the other hand, if $\mathbf{u}\in L^q(\Omega, \mathbb{R}^d)$ and
$\int_\Omega \mathbf{u}\cdot \mathbf{v}\, dx=0$
for any $\mathbf{v}\in C_\sigma^\infty(\Omega)$,
then $\mathbf{u}=-\nabla \psi$ for some $\psi \in L_{loc}^1(\Omega)$
(see e.g. \cite[pp.696-697]{FuMo} for a proof).
This implies that $\mathbf{u}\in \text{grad } W^{1,q}(\Omega)$.
Hence we obtain grad $W^{1,q}(\Omega)= (X^p(\Omega))^\perp$.
It follows that $X^p(\Omega)= (\text{grad } W^{1,q}(\Omega))^\perp
=L^p_\sigma (\Omega)$
and thus $C_\sigma^\infty(\Omega)$ is dense in $L^p_\sigma (\Omega)$.
\end{proof}

\begin{lemma}\label{lemma-6.2}
Let $\Omega$ be a bounded Lipschitz domain in $\mathbb{R}^d$
and $1<p<\infty$. If the Helmholtz decomposition (\ref{H-decomposition}) with the estimate
(\ref{H-estimate}) holds 
for the exponent $p$ and constant $C_p$, then it holds
for the dual exponent $q=\frac{p}{p-1}$ and constant $C_q=C_p$.
\end{lemma}

\begin{proof} This follows from the fact that if $X_0$, $X_1$ are closed subspaces of $X$
and $X=X_0\oplus X_1$, then
$X^*=X_0^\perp \oplus X_1^\perp$.
Note that if $\mathbf{u}\in L^p_\sigma(\Omega)$, $\mathbf{v}\in L^q_\sigma(\Omega)$,
$\phi\in W^{1,p}(\Omega)$ and $\psi\in W^{1,q}(\Omega)$, then
\begin{equation}
\aligned
\int_\Omega \mathbf{u} \cdot (\mathbf{v}+\nabla \psi)\, dx  & =\int_\Omega (\mathbf{u} 
+\nabla \phi)\cdot \mathbf{v}\, dx,\\
\int_\Omega \nabla \phi \cdot (\mathbf{v} +\nabla \psi)\, dx
 & =\int_\Omega (\mathbf{u} +\nabla \phi) \cdot \nabla \psi\, dx.
\endaligned
\end{equation}
By a simple duality argument, this shows that the estimate (\ref{H-estimate}) holds for
the exponent $q$ and constant $C_q=C_p$.
\end{proof}

Next we will show that the $L^p$-Helmholtz decomposition is equivalent to the solvability of 
(\ref{B-Neumann-problem}) for $s=1/p$.

\begin{thm}\label{equiv-theorem}
Let $\Omega$ be a bounded Lipschitz domain in $\mathbb{R}^d$ and $1<p<\infty$.
Then the $L^p$-Helmholtz decomposition (\ref{H-decomposition}) with the estimate 
(\ref{H-estimate}) holds
if and only if the Neumann problem (\ref{B-Neumann-problem}) is
uniquely solvable for $s=1/p$.
\end{thm}

\begin{proof}
Suppose that (\ref{B-Neumann-problem}) is uniquely solvable for $s=1/p$.
The uniqueness of the solutions to (\ref{B-Neumann-problem}) 
implies that $L^p_\sigma (\Omega)\cap \text{grad } W^{1,p}(\Omega)=\{ 0\}$.
Given any $\mathbf{u}=(u_1, \dots, u_d)\in L^p(\Omega, \mathbb{R}^d)$, let
\begin{equation}
\phi (x)=\frac{\partial}{\partial x_i}
\int_\Omega \Gamma (x-y) u_i (y)\, dy,
\end{equation}
where $\Gamma(x)$ denotes the fundamental solution for $\Delta$ in $\mathbb{R}^d$
with pole at the origin.
By the Calder\'on-Zygmund estimate, $\phi\in W^{1,p}(\Omega)$ and $\| \nabla \phi\|_{L^p(\Omega)}
\le C\| \mathbf{u}\|_{L^p(\Omega)}$. Since
$\text{div}(\mathbf{u}-\nabla \phi)=0$ in $\Omega$,
it follows that $\Lambda =(\mathbf{u}-\nabla \phi)\cdot n \in B^p_{-1/p}(\partial\Omega)$
and $\|\Lambda\|_{B^p_{-1/p}(\partial\Omega)} \le C \| \mathbf{u}-\nabla \phi\|_{L^p(\Omega)}$,
where $\Lambda$ may be defined by
$$
<\Lambda, \varphi>=\int_\Omega (\mathbf{u}-\nabla \phi)\cdot \nabla \widetilde{\varphi}\, dx
$$
for $\varphi \in B^q_{1/p}(\partial\Omega)$ and $\widetilde{\varphi}\in W^{1,q}(\Omega)$
such that $Tr(\widetilde{\varphi})=\varphi$ on $\partial\Omega$.
We now let $\mathbf{v}=\mathbf{u}- \nabla (\phi+\psi)\in L^p(\Omega, \mathbb{R}^d)$, where
$\psi\in W^{1,p}(\Omega)$ is a solution to (\ref{B-Neumann-problem}) with boundary data
$\Lambda$.
Observe that for any $\varphi \in C^\infty (\mathbb{R}^d)$,
$$
\int_\Omega \mathbf{v}\cdot \nabla \varphi \, dx
=<\Lambda,\varphi>-\int_\Omega \nabla \psi\cdot \nabla \varphi=0.
$$
Thus $\mathbf{v}\in L^p_\sigma(\Omega)$. Also note that
$$
\aligned
\| \nabla (\phi +\psi)\|_{L^p(\Omega)}
 & \le C\,  \big\{ \| \nabla \phi \|_{L^p(\Omega)}
+\| \Lambda \|_{B^p_{-1/p}(\partial\Omega)}\big\}\\
 & \le C\,  \big\{ \| \nabla \phi \|_{L^p(\Omega)}
+\| \mathbf{u} -\nabla \phi\|_{L^p(\Omega)}\big\}\\
& \le C \, \| \mathbf{u}\|_{L^p(\Omega)}.
\endaligned
$$
Since $\mathbf{u}=\mathbf{v}+\nabla (\phi +\psi)$, 
we obtain the Helmholtz decomposition (\ref{H-decomposition-1}).

Next suppose that the $L^p$-Helmholtz decomposition (\ref{H-decomposition}) with
estimate (\ref{H-estimate}) holds.
The uniqueness for the Neumann problem (\ref{B-Neumann-problem}) follows
from the fact that $L^p_\sigma (\Omega)\cap \text{grad } W^{1,p}(\Omega)$
$=\{ 0 \}$.
To show the existence, let $\psi$ be a solution of the $L^2$ Neumann problem in $\Omega$
with boundary data $\frac{\partial \psi}{\partial n}=g$, 
where $g\in L^\infty(\partial\Omega)$ and $\int_{\partial\Omega} g
d\sigma =0$.
Given $\mathbf{u}\in L^q (\Omega, \mathbb{R}^d)\cap L^2 (\Omega, \mathbb{R}^d)$,
write $\mathbf{u}=\mathbf{v} +\nabla \phi$, where
$\mathbf{v}\in L^q_\sigma (\Omega)\cap L^2_\sigma (\Omega)$
and $\phi \in W^{1,q}(\Omega)\cap W^{1,2}(\Omega)$.
This is possible since the Helmholtz decomposition holds for exponents $q$ and $2$.
It follows that
$$
\aligned
\left| \int_\Omega \nabla \psi \cdot \mathbf{u}\, dx\right|
& =\left| \int_\Omega \nabla \psi \cdot \nabla \phi \, dx\right|
=\left| \int_{\partial\Omega}
\frac{\partial \psi}{\partial n}
\cdot (\phi-\alpha)\, d\sigma\right| \\
&\le \|\frac{\partial \psi }{\partial n}\|_{B^p_{-1/p}(\partial\Omega)}
\| \phi-\alpha\|_{B^q_{1/p}(\partial\Omega)}\\
& \le \|\frac{\partial \psi }{\partial n}\|_{B^p_{-1/p}(\partial\Omega)}
\| \phi-\alpha \|_{W^{1,q}(\Omega)},
\endaligned
$$
for any $\alpha \in \mathbb{R}$. By Poincar\'e inequality, this yields that
$$
\left| \int_\Omega \nabla \psi \cdot \mathbf{u}\, dx\right|
\le C \|\frac{\partial \psi }{\partial n}\|_{B^p_{-1/p}(\partial\Omega)}
\|\nabla \phi\|_{L^q(\Omega)}.
$$
Using $\|\nabla \phi \|_{L^q(\Omega)}
\le C_q \|\mathbf{u}\|_{L^q(\Omega)}$, we then obtain 
\begin{equation}\label{density-estimate}
\|\nabla \psi\|_{L^p(\Omega)} \le C 
\|\frac{\partial \psi }{\partial n}\|_{B^p_{-1/p}(\partial\Omega)}
\end{equation}
by duality.
Since $L^\infty(\partial\Omega)$ is dense in $B^p_{-1/p}(\partial\Omega)$,
the existence of solutions to (\ref{B-Neumann-problem}) with data $\Lambda$, where
$\Lambda\in B^p_{-1/p}(\partial\Omega)$ and $<\Lambda, 1>=0$,
follows from the estimate (\ref{density-estimate}) by a simple limiting argument.
This completes the proof.
\end{proof}

Lemma \ref{lemma-6.2} and Theorem \ref{equiv-theorem} lead to the following.

\begin{thm}\label{equiv-theorem-1}
Let $\Omega$ be a bounded Lipschitz domain in $\mathbb{R}^d$ and $1<p<\infty$.
Then the solvability of
(\ref{B-Neumann-problem}) for $s=1/p$ is equivalent to the solvability
of (\ref{B-Neumann-problem}) for $s=1/q$, where $q=\frac{p}{p-1}$.
\end{thm}

\begin{remark}\label{Remark-6.3}
{\rm
We show in Section 5 that if $\Omega$ is a bounded convex domain in $\mathbb{R}^d$,
then the Neumann problem (\ref{B-Neumann-problem}) is solvable for $s=1/p$ and $2<p<\infty$.
Thus, by Theorem \ref{equiv-theorem-1}, the Neumann problem (\ref{B-Neumann-problem})
in convex domains with $s=1/p$ is solvable for any $1<p<\infty$.
This completes the proof of Theorem \ref{cor-1}.
}
\end{remark}

\begin{remark}\label{Remark-6.4}
{\rm
Theorem \ref{cor-2} follows readily from Theorems \ref{cor-1} and \ref{equiv-theorem}.
}
\end{remark}

\bibliography{gs}

\small
\noindent\textsc{Department of Mathematics, University of Kentucky, Lexington, KY 40506}\\
\emph{E-mail address}: \texttt{jgeng@ms.uky.edu} \\

\noindent\textsc{Department of Mathematics, 
University of Kentucky, Lexington, KY 40506}\\
\emph{E-mail address}: \texttt{zshen2@email.uky.edu} \\

\noindent \today
\end{document}